\documentclass[12pt]{amsart}
\usepackage{amsmath}
\usepackage{amsfonts}
\usepackage{latexsym}
\usepackage{amssymb}
\usepackage[dvips]{graphics}
\newtheorem{theorem}{Theorem}
\newtheorem{proposition}[theorem]{Proposition}
\newtheorem{lemma}[theorem]{Lemma}
\newtheorem{corollary}[theorem]{Corollary}

\newtheorem{example}{Example}
\newcommand{\C}{{\mathbf C}}
\newcommand{\R}{{\mathbf R}}
\newcommand{\Z}{{\mathbf Z}}

\newcommand{\Hom}{{\rm {Hom}}}
\newcommand{\rk}{{\sf {rk}}}
\begin{document}

\title[Homology of telescopic linkages]
 {Homology of planar telescopic linkages}

\author{Michael Farber and Viktor Fromm}        % Enter your name between curly braces
\date{\today}          % Enter your date or \today between curly braces

\address{Department of Mathematical Sciences, University of Durham, UK}

\email{Michael.Farber@durham.ac.uk}
\email{viktor.fromm@durham.ac.uk}

%\thanks{This work was completed with the support of an Izaak Walton Killam Memorial Scholarship.}

%\thanks{The author was also supported in part by the Research Council of Slovenia.}



\dedicatory{}

% -----------------------------------------------------------

\begin{abstract}
We study topology of configuration spaces of planar linkages having one leg of variable length.
Such telescopic legs are common in modern robotics where they are used for shock absorbtion and serve
a variety of other purposes.
Using a Morse theoretic technique, we compute explicitly, in terms of the metric data, the Betti numbers of configuration spaces of these mechanisms.
\end{abstract}

% -----------------------------------------------------------
\maketitle
% -----------------------------------------------------------

\section {Introduction}
A planar linkage is a mechanism shown on Figure \ref{fig1}; it consists of several bars of fixed length connected by revolving joints forming a closed polygonal chain; the positions of two adjacent vertices are fixed but the other vertices are free to move in the plane.
\begin{figure}[t]
\begin{center}
\resizebox{6cm}{4.5cm}{\includegraphics[115,391][431,608]{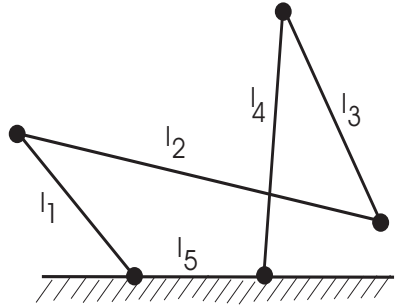}}
\end{center}
\caption{A planar linkage.}\label{fig1}
\end{figure}
The configuration space of a planar linkage depends on the bar lengths $\ell_1, \dots, \ell_n$ and is generically a closed smooth manifold of dimension $n-3$ where $n$ is the number of bars in the mechanism. For some special collections of bar lengths
$\ell_1, \dots, \ell_n$ the configuration space is a compact manifold with finitely many singular points, see for instance \cite{Fa3}.

Configuration spaces of planar linkages appear also as moduli spaces of shapes of planar $n$-gons with prescribed side lengths. These same manifolds emerge in statistical shape theory \cite{KBC}; they also describe spaces of stable and semi-stable configurations of labeled points on the projective plane which play an important role in algebraic geometry and mathematical physics.

Mathematical study of linkages and more general mechanisms has
a long history going back to the Middle Ages. Engineering discoveries
involving linkages played an important role in the industrial revolution.
Topological theory of linkages was initiated by W. Thurston,
his students and collaborators.
Kevin Walker \cite{Wa} in his 1985 Princeton undergraduate thesis gives
an amazingly deep picture of configuration spaces of linkages.
A. A. Klyachko
\cite{Kl} used methods of algebraic geometry to find an explicit expression for the Betti numbers of
configuration spaces of linkages in $\R^3$. Betti numbers of planar linkages were fully described in \cite{Fa2}; the result of \cite{Fa2} covers also the non-generic cases.
Significant progress in topology of linkages was made by J.-
Cl. Hausmann, A. Knutson, M. Kapovich and J. Millson \cite{Hau1, Hau2} and  \cite{Ka1}.
Non-generic
polygon spaces were independently studied by the Japanese
school (see, e.g. \cite{Kam1}).%, Kam2}).

Monograph \cite{Fa3} contains a detailed exposition of the topology of linkages.
We also refer the reader to the book \cite{DR} providing a wealth of information about linkages and their applications in engineering.

In this paper we study a planar mechanism which is slightly more general than the usual planar linkage. Namely, we assume that there are $n$ bars connected cyclically as shown on Figure \ref{fig1} and all bars except one have constant lengths; however the remaining bar is assumed to be {\it telescopic}, i.e. its length may vary in a prescribed interval $[a, b]$ where $a\le b$. Telescopic legs are quite common in modern robotics; they serve many practical purposes, for example they are used for shock absorbtion.

The subject of this article, besides its obvious importance for the theory of mechanisms and for the control theory, carries special charm of
vigorous interplay of tools belonging to very different branches of mathematics: topology of manifolds (in particular, Morse theory),
group actions, and combinatorics. Symmetry enters the game in the form of various involutions which are important as they imply perfectness of certain Morse functions (see \cite{Fa2}, \cite{Fa3}). The crucial role plays combinatorics of short and long subsets leading to a decomposition of the simplex of length parameters into chambers encoding the topological types of generic configuration spaces of linkages \cite{FHS}.

\section{Configuration space of linkage with telescopic leg}\label{sec2}

In order to give a formal definition of the configuration space of a linkage with one telescopic leg consider the following continuous map
\begin{eqnarray}\label{eff}
F: \C^n \to \R^n, \quad F(z_1, z_2, \dots, z_n) = (\ell_1, \ell_2, \dots, \ell_n),
\end{eqnarray}
where
\begin{eqnarray}\label{two}
\ell_i= |z_{i+1}-z_{i}|, \quad i=1, \dots, n.
\end{eqnarray}
The indices in (\ref{two}) are understood cyclically modulo $n$, i.e. $z_{n+1}=z_1$.
Let $E(2)$ denote the group of orientation preserving isometries of the plane $\C=\R^2$. The
map $F$ is invariant under the diagonal action of
$E(2)$ on $\C^n$.
If $\ell=(\ell_1, \ell_2, \dots, \ell_n)\in \R^n$ is a prescribed length vector, $\ell_i>0$, then
\begin{eqnarray}\label{link}
M_\ell = F^{-1}(\ell)/E(2)
\end{eqnarray}
is the moduli space of shapes of planar $n$-gons with sides having lengths $\ell_1, \dots, \ell_n$.

Let us now assume that we have two length vectors $\ell^\pm =(\ell^\pm_1, \dots, \ell^\pm_n)$ where
$$\ell^-_j=\ell^+_j=\ell_j>0 \quad \mbox{for all $j\in \{1, \dots, n-1\}$}$$ and $$ 0<\ell^-_n<\ell^+_n.$$ Here
$n$ is the index corresponding to the telescopic leg: we assume that the length of the $n$-th bar is not fixed but is variable in the segment\footnote{In this paper we always assume that the lower bound for the length of the telescopic leg is positive, $\ell^-_n>0$, and we do not allow $\ell_n^-=0$.}
 $[\ell_n^-, \ell_n^+]$. We consider the interval of length vectors
$A\subset \R^n$ which is parallel to the $n$-th axis and connects the vectors $\ell^-$ and $\ell^+$:
 $$A=\{\ell=(\ell_1, \dots, \ell_n); \, \ell_j^-\le \ell_j \le \ell_j^+, \, 1\le j\le n\}.$$ The configuration space of a linkage with a telescopic leg is defined similarly to (\ref{link}) as
\begin{eqnarray}\label{link1}
M_A = F^{-1}(A)/E(2).
\end{eqnarray}
The symbol $A$ in the notation $M_A$ can be viewed as representing all metric data of a telescopic linkage.

We will say that a metric data $A$ as above is {\it generic} if
$$\sum_{j=1}^n \epsilon_j \ell_j^-\not= 0, \quad \mbox{and}\quad \sum_{j=1}^n \epsilon_j \ell_j^+\not= 0$$
for any choice of coefficients $\epsilon_j=\pm 1$.

\begin{proposition}\label{lm1} If $A$ is generic then $M_A$ is a smooth compact orientable manifold with boundary and  $\dim M_A = n-2$. The boundary of $M_A$ is a disjoint union of the manifolds $M_{\ell^-}$ and $M_{\ell^+}$.
\end{proposition}
The proof is given in section \S \ref{proofs}.

For $1\le j\le n$ let $H_{j}\subset \C^n$ denotes the hyperplane $z_j=z_{j+1}$ (note that the hyperplane $H_n$ is given by the equation $z_n=z_1$). The map $F$ (see equation (\ref{eff})) is smooth when restricted onto the complement $$X=\C^n-\cup_{j}H_{j}.$$
It is well-known that the critical points of $F|X$ are collinear configurations, i.e. the collections $(z_1, \dots, z_n)\in \C^n$ such that the points $z_1, \dots, z_n$ lie on an affine real line $L$ in $\C$. The critical values of $F|X$ are vectors $(\ell_1, \dots, \ell_n)$ corresponding to
collinear configurations. If $(z_1, \dots, z_n)$ is a collinear configuration lying on an affine real line $L$ then $z_i-z_{i+1} = \epsilon_i\ell_iv$, where $\ell_i=|z_i-z_{i+1}|$, $v$
is a fixed unit vector parallel to $L$ and $\epsilon_i=\pm 1$. Then $\sum_{j=1}^n \epsilon_j\ell_j=0$ and thus the set of critical values of $F|X$ equals
$$\left(\bigcup_J S_J\right) \bigcap \R^n_+.$$
Here the symbol $J$ runs over all proper subsets $J\subset \{1, \dots, n\}$ and $S_J\subset \R^n$ denotes the hyperplane
$$\sum_{j\in J}\ell_j = \sum_{j\notin J}\ell_j.$$

For a length vector $\ell=(\ell_1, \dots, \ell_n) \in \R^n_+$ we denote by $[\ell]$ the number
\begin{eqnarray}\label{five}
[\ell] = \min \left(\sum_{i=1}^n \epsilon_i \ell_i\right)\end{eqnarray}
where for $i=1, \dots, n$ the numbers $\epsilon_i=\pm 1$ are such that
$\sum_{i=1}^n \epsilon_i \ell_i\ge 0$.
Clearly $[\ell]$ is a measure of \lq\lq genericity\rq\rq of the vector $\ell$; indeed, $[\ell]\not=0$ if and only if $\ell$ is generic.

\begin{proposition}\label{cor2} Consider a telescopic linkage with generic metric data $A$ consisting of numbers $\ell_1, \dots, \ell_{n-1}$ and parameters of the telescopic leg $\ell^-_n<\ell_n^+$. Suppose that the difference $\ell_n^+-\ell_n^-$ satisfies
\begin{eqnarray}\label{ineq3}
\ell^+_n-\ell_n^-< [\ell^-].\end{eqnarray}
Then
$M_A$ is diffeomorphic to the Cartesian product
$$M_A\simeq M_{\ell^-}\times [0,1],$$
where $\ell^-=(\ell_1, \dots, \ell_{n-1}, \ell_n^-)$.
\end{proposition}
\begin{proof} The statement of Proposition \ref{cor2} follows once we know that inequality (\ref{ineq3}) guarantees that the interval of length vectors $$\{(\ell_1, \dots, \ell_n); \, \ell_n^-\le \ell_n\le \ell_n^+\}\subset \R^n$$ does not cross any of the hyperplanes $S_J$, see above.
Suppose that for some $\ell\in A$ the length vector $\ell$ lies in $S_J$. Then $\sum_{i=1}^n \epsilon_i \ell_i=0$ where $\epsilon_i=1$ if $i\in J$ and $\epsilon_i=-1$ if $i\notin J$. Without loss of generality we may assume that $\epsilon_n=1$. Then one has
$$\sum_{i=1}^{n-1} \epsilon_i\ell_i +\ell_n^-<0, \quad \sum_{i=1}^{n-1} \epsilon_i\ell_i +\ell_n^+>0.$$
It follows that $$\sum_{i=1}^{n-1}\epsilon_i\ell_i +\ell_n^-\le -[\ell^-]$$ and
$$0<(\ell_n^+-\ell_n^-)+\left(\ell_n^- +\sum_{i=1}^{n-1}\epsilon_i\ell_i\right) \le (\ell^+_n-\ell^-_n)-[\ell^-],$$
contradicting (\ref{ineq3}).
\end{proof}

%\begin{corollary}\label{cor2} If the interval $A\subset \R^n$ does not intersect any of the hyperplanes $S_J$ then $M_A$ is diffeomorphic to the Cartesian product
%$M_A=M_{\ell^-}\times [0,1]$.
%\end{corollary}

%The condition of Corollary \ref{cor2} can be expressed in metric terms as follows. Define
%$${\mathbf m}(\ell^-)=\min\{\sum_{i=1}^{n-1} \ell_i\epsilon_i; \, \epsilon_i=\pm1, \, \sum_{i=1}^{n-1} \ell_i\epsilon_i>\ell^-_n\}.$$
%Then Corollary \ref{cor2} applies if and only if $\ell_n^+< {\mathbf m}(\ell^-)$.

By symmetry, one may always assume
without loss of generality that
\begin{eqnarray}\label{usual}
\ell_1\le \ell_2\le \dots\le \ell_{n-1}.
\end{eqnarray}
However the interval $[\ell^-_n, \ell^+_n]$ may interact with the sequence of numbers (\ref{usual}) in various ways.
\begin{proposition}\label{empty} Under the condition (\ref{usual}) the manifold $M_A$ is nonempty if and only if the intervals $[\ell^-_n, \ell^+_n]$ and $[r, R]$ have a nonempty intersection. Here
$R=\ell_1+\dots+\ell_{n-1}$ and $r=\ell_{n-1}-\ell_1-\dots -\ell_{n-2}$.
\end{proposition}
\begin{proof} Assume that $M_A\not=\emptyset$, and let $(z_1, \dots, z_n)$ be a configuration with $|z_{i+1}-z_i|= \ell_i$ for $i=1, \dots, n-1$ and
$$\ell_n^-\le |z_n-z_1|\le \ell_n^+.$$ Then clearly, using the triangle inequality,
$$\ell_{n-1}-\ell_1-\dots -\ell_{n-2} \le |z_n-z_1|\le \ell_1+ \dots+\ell_{n-1}.$$
Hence $[\ell_n^-, \ell_n^+]\cap [r,R]\not=\emptyset$.

Conversely, suppose that $\rho\in [\ell_n^-, \ell_n^+]\cap [r,R]$. Then there exists a configuration of points $(z_1, \dots, z_n)\in \C^n$ such that
$|z_{i+1}-z_i|= \ell_i$ for $i=1, \dots, n-1$ and $|z_n-z_1|=\rho.$ Hence $M_A\not=\emptyset.$
\end{proof}

\section{Betti numbers of $M_A$}

In this section we state the main theorem of this paper which gives explicitly the Betti numbers of manifolds $M_A$.

Recall that $A$ denotes the metric data of the telescopic linkage consisting of two vectors $\ell^+, \ell^-\in \R^n_+$ which have all coordinates equal $\ell^+_i=\ell^-_i=\ell_i>0$ for all $i<n$ and $\ell^+_n>\ell^-_n>0$. In other words the telescopic leg corresponds to the $n$-th coordinate.
We will also assume the inequalities (\ref{usual}).

Before stating our main result we have to define some combinatorial quantities.
For a subset $J\subset \{1, \dots, n\}$ one denotes by $\epsilon_J=(\epsilon_1, \dots, \epsilon_n) \in \R^n$ the vector having
coordinates $\epsilon_i=1$ if $i\in J$ and $\epsilon_i=-1$ if $i\notin J$. One may view the vectors $\epsilon_J$ for various $J$ as vertexes of the unit cube
$C=[-1,1]^n\subset \R^n$.

Given $\ell\in \R^n_+$ and an integer $k = 0, 1, \dots, n-2$ we denote by $\alpha_k(\ell)$ the number of subsets $J\subset \{1, \dots, n-1\}$ of cardinality
$|J|=n-k-1$ such that $\langle \ell, \epsilon_J\rangle >0$. The last inequality may also be expressed by saying\footnote{According to a well established terminology a subset $J\subset \{1, \dots, n\}$ is called long with respect to a length vector $\ell$ if $\langle \epsilon_J, \ell\rangle >0$. A subset $J$ is called short with respect to $\ell$ if its complement is long.} that \lq\lq $J$ is long with respect to $\ell$\rq\rq.

Passing to complements, we see that $\alpha_k(\ell)$ equals the number of $k+1$ element subsets of the index sets $\{1, \dots, n\}$ which contain $n$ (the index of the telescopic leg),  and are
short with respect to $\ell$.

Given two vectors $\ell^+, \ell^-\in \R_+^n$ with $\ell^+_i=\ell^-_i=\ell_i$ for $i=1, \dots, n-1$ and an integer $k=0, \dots, n-2$, we denote by $\beta_k(\ell^+, \ell^-)$ the number of subsets $J\subset \{1, \dots, n-2\}$ of cardinality $|J|= n-k-2$ such that
\begin{eqnarray}
\langle \ell^+, \, \epsilon_{J'}\rangle <0, \quad \mbox{and}\quad \langle \ell^- ,\, \epsilon_{{J''}}, \rangle >0
\end{eqnarray}
where $J'=J\cup \{n\}$ and $J''=J\cup \{n-1\}$. In other words, $J'$ is short with respect to $\ell^+$ and $J''$ is long with respect to $\ell^-$.

Each subset $J$ as above determines a subset $K\subset \{1, \dots, n\}$ (the complement of $J''$ in $\{1, \dots, n\}$) which has the following properties:
\begin{enumerate}
\item[(a)] $|K|=k+1$;
\item[(b)] $n\in K$ and $n-1\notin K$;
\item[(c)] $K$ is short with respect to $\ell^-$;
\item[(d)] The set $K'$ obtained from $K$ by removing $n$ and adding $n-1$ is long with respect to $\ell^+$.
\end{enumerate}

\noindent Clearly $\beta_k(\ell^+, \ell^-)$ equals the number of subsets $K$ satisfying (a) - (d).

Note the following symmetry property:
\begin{eqnarray}\label{dual}
\beta_k(\ell^+, \ell^-) = \beta_{n-2-k}(\ell^-, \ell^+),
\end{eqnarray}
which follows by passing to complements of subsets and adding $n$, i.e. by considering the map $K\mapsto \bar K\cup\{n\}$. Next we observe that
\begin{eqnarray}
\alpha_k(\ell^-) \ge \beta_k(\ell^+, \ell^-).
\end{eqnarray}
We also mention the following property:
\begin{lemma} Assume that the average length of the telescopic leg
is longer than any other leg of the linkage, i.e.
\begin{eqnarray}
\frac{\ell_n^++\ell_n^-}{2} \ge \ell_{n-1}.
\end{eqnarray}
Then $\beta_k(\ell^+, \ell^-) =0$ for all $k$.
\end{lemma}
\begin{proof} Assume that $\beta_k(\ell^+, \ell^-) \not=0$, i.e. there exists a subset $K$ satisfying (a) - (d). Denote
$$x= \sum_{i\in K, \, i\not= n}\ell_i - \sum_{i\notin K} \ell_i.$$
We have two inequalities
$\ell^-_n +x<0$
(because of (b) and (c)) and
$2\ell_{n-1}-\ell^+_n+x>0$
(because of (d)). These two inequalities imply that $2\ell_{n-1}>\ell_n^++\ell_n^-$ contradicting our assumption.
\end{proof}

The following statement is the main result of this paper.

\begin{theorem}\label{main} Let $A$ be the metric data of a telescopic linkage having legs of fixed lengths $\ell_1\le \ell_2\le \dots \le \ell_{n-1}$ and a telescopic leg of length varying
between $\ell^-_n$ and $\ell^+_n$, where $0<\ell^-_n<\ell^+_n$. Assume that the metric data $A$ is generic, see above. Then the homology group $H_k(M_A;\Z)$ is free abelian and its rank equals
\begin{eqnarray}
\alpha_k(\ell^-) - \beta_k(\ell^+, \ell^-) +\alpha_{n-3-k}(\ell^+)  - \beta_{n-3-k}(\ell^-, \ell^+)
\end{eqnarray}
for $k=0, \dots, n-2$.
\end{theorem}

\section{Examples}

Before embarking on the proof of Theorem \ref{main} in the next section we consider a few special cases.

\begin{example} {\rm
Suppose that the numbers $\ell_n^+$ and $\ell_n^- $ are nearly equal. In this case the manifold $M_A$ is diffeomorphic to the product $M_\ell\times [0,1]$ where $M_\ell$ is the moduli space of closed linkage with length vector $\ell=(\ell_1, \dots, \ell_n)$ where $\ell_i=\ell_i^+$ for all $i$. We want to compare the statement of Theorem \ref{main} in this special case with the result of \cite{Fa2} giving Betti number of planar linkages with a fixed length vector.
It is known \cite{Fa2} that the integral homology groups of planar polygon spaces are free abelian and therefore their Betti numbers are independent of the field of coefficients.
Set $\ell^+=\ell^-= \ell$ and consider the difference
$$\alpha_k(\ell^-) -\beta_k(\ell^+, \ell^-)=\alpha_k(\ell) -\beta_k(\ell, \ell).$$ Without loss of generality we may assume that $\ell_1\le \ell_2 \le \dots \le \ell_{n-1}$ however the last coordinate $\ell_n$ (corresponding to the telescopic leg) can be arbitrary.

According to our definition, the number $\alpha_k(\ell)$ is the number of subsets of the set
$\{1, \dots, n-1\}$  which are of cardinality $n-k-1$ and are long with respect to $\ell$. Passing to complements, we see that $\alpha_k(\ell)$ equals the number of subsets
of $\{1, \dots, n\}$ of cardinality $k+1$ which  contain $n$ and
are short with respect to $\ell$.

The other number $\beta_k(\ell, \ell)$ equals the number of $J\subset \{1, \dots, n-2\}$ with $|J|=n-k-2$ such that $J'=J\cup\{n\}$ is short with respect to $\ell$ and
$J''=J\cup\{n-1\}$ is long with respect to $\ell$. Each such subset $J$ determines a subset $K\subset \{1, \dots, n\}$ (the complement of $J''$ in $\{1, \dots, n\}$) which has the following properties: (a) $|K|=k+1$; (b) $n\in K$ and $n-1\notin K$; (c) $K$ is short with respect to $\ell$; (d) The set $K'$ obtained from $K$ by removing $n$ and adding $n-1$ is long with respect to $\ell$. Clearly $\beta_k(\ell, \ell)$ equals the number of subsets $K$ satisfying properties (a) - (d).

Consider now two cases.

(I) If $\ell_{n-1}\le \ell_n$ then obviously $\beta_k(\ell, \ell)=0$ and the number $\alpha_k(\ell)$ coincides with the number $a_k(\ell)$ defined in \cite{Fa2} as the number of short subsets of cardinality $k+1$ containing the index of the longest link $n$.

(II) Assume now that $\ell_{n-1}>\ell_n$. The number $\alpha_k(\ell)$ equals the number of short subsets of cardinality $k+1$ containing $n$. The family of all subsets of cardinality $k+1$ which contain $n$ and are short with respect to $\ell$ can be represented as the union of three mutually disjoint families $$A\cup B\cup C,$$ where $A$ is the family of all subsets
$K\subset \{1, \dots, n\}$ of cardinality $k+1$ with $n-1, n\in K$ which are short with respect to $\ell$;
$B$ is the family of all subsets
$K\subset \{1, \dots, n\}$ of cardinality $k+1$ with $ n\in K$ and $n-1\notin K$ such that $K$ and $\check{K}=K-\{n\}\cup \{n-1\}$ are short with respect to $\ell$;
 $C$ is the family of all
 subsets $K\subset \{1, \dots, n\}$ of cardinality $k+1$ with $ n\in K$ and $n-1\notin K$ such that $K$ is short and
 $\check{K}=K-\{n\}\cup \{n-1\}$ is long with respect to $\ell$.

 Clearly $\beta_k(\ell, \ell)$ is exactly the cardinality of $C$. Hence the difference
$\alpha_k(\ell)-\beta_k(\ell, \ell)$ equals $a_k(\ell)$ as defined in \cite{Fa2}, the number of short subsets of cardinality $k+1$ containing $n-1$, i.e. the index of the longest link.

Thus we see that Theorem \ref{main} implies  Theorem 1 from \cite{Fa2} in the nonsingular case (note that the latter results covers also the cases when the moduli space of linkages has singularities).

}

\begin{example}{\rm
Assume that (a) $\ell_{n-1}>\ell_1+\dots+\ell_{n-2}$; (b) $\ell_n^->0$ is very small; and (c) $\ell_n^+>\ell_1+\dots+\ell_{n-1}$ is very large. Then clearly $M_A=T^{n-2}$
is the $(n-2)$-dimensional torus. To apply Theorem \ref{main} one computes the numbers $\alpha_k(\ell^-)$ and $\alpha_k(\ell^+)$. A subset $J\subset \{1, \dots, n-1\}$ is long with respect to $\ell^-$ if and only if it contains $n-1$. There are no subsets $J\subset \{1, \dots, n-1\}$ which are long with respect to $\ell^+$. Thus we obtain
$\alpha_k(\ell^-)= \binom {n-2} k$ and $\alpha_k(\ell^+)=0$. The numbers $\beta_k$ all vanish in this case. We see that the result is consistent with the fact that
$M_A=T^{n-2}$.
}\end{example}
\end{example}
\begin{example}{\rm Consider the zero-dimensional Betti number as given by Theorem \ref{main}. Analyzing the definitions given above one sees that the difference
$\alpha_0(\ell^-)-\beta_0(\ell^+, \ell^-)$ can be either $0$ or $1$ and it equals $1$ if and only if the following inequalities hold
$$\ell^-_n<\ell_1+ \dots+ \ell_{n-1}\quad \mbox{and}\quad \ell^+_n>\ell_{n-1} -\ell_1 -\ell_2-\dots-\ell_{n-2}.$$
Denoting $R=\ell_1+ \dots+ \ell_{n-1}$ and $r= \ell_{n-1} -\ell_1 -\ell_2-\dots-\ell_{n-2}$, we may express the above two inequalities equivalently as
$[\ell^-_n, \ell^+_n]\cap [r, R]\not=\emptyset.$
It follows that $\alpha_0(\ell^-)-\beta_0(\ell^+, \ell^-)$ equals one if and only if the manifold $M_A$ is nonempty, see Proposition \ref{empty}.

Note that in general the difference $\alpha_{k}(\ell^+) - \beta_{k}(\ell^-, \ell^+)$ equals the number of subsets $J\subset \{1, \dots, n-1\}$ with $|J|=n-k-1$ such that
$J$ is long with respect to $\ell^+$ and either $n-1\notin J$ or $n-1\in J$ and the set $J\cup \{n\}-\{n-1\}$ is long with respect to $\ell^-$.

Substituting $k=n-3$, we obtain that $\alpha_{n-3}(\ell^+) - \beta_{n-3}(\ell^-, \ell^+)$ equals the number of two-element subsets $J\subset \{1, \dots, n-1\}$ which are  long with respect to $\ell^+$ and either (a) $n-1\notin J$ or (b) $n-1\in J$ and the set $J\cup \{n\}-\{n-1\}$ is long with respect to $\ell^-$. If (a) occurs then clearly $J=\{n-3, n-2\}$ and $\ell_n^+\le \ell_{n-3}$;
the necessary and sufficient condition for (a) is given by the inequality
\begin{eqnarray}\label{a}
2(\ell_{n-3}+\ell_{n-2})\ge \ell_1+\dots+\ell_{n-1}+\ell_n^+.
\end{eqnarray}
We see that there may be at most one set $J$ satisfying (a).

Suppose now that (b) is satisfied. Then the subset $J$ must coincide with $\{n-2, n-1\}$ since for any other choice
$J=\{i, n-1\}$ (with $i<n-2$) we would have the sets $\{n-2, n-1\}$ and $\{i, n\}$ long and mutually disjoint with respect to $\ell^-$, which is impossible. Hence the case (b) is equivalent to the inequalities
\begin{eqnarray}\label{b1}
2(\ell_{n-2}+\ell_{n-1})\ge \ell_1+\dots+\ell_{n-1}+\ell_n^+,\end{eqnarray}
and
\begin{eqnarray}\label{b2}
2(\ell_{n-2}+\ell^-_{n})\ge \ell_1+\dots+\ell_{n-1}+\ell_n^-.\end{eqnarray}
This last inequality implies that $\{n-2, n\}$ is long with respect to $\ell^+$ which is inconsistent with $\{n-3, n-2\}$ being long with respect to $\ell^+$, i.e. with the case (a). Indeed,  if $\{n-2, n\}$ is long then $\{n-1, n\}$ is long and we obtain any subset lying in the complement of $\{n-1, n\}$ (such as $\{n-3, n-2\}$) is short.

We obtain that the cases (a) and (b) are inconsistent with each other and either of the cases is satisfied by at most one subset.

\begin{corollary}\label{cor6} The manifold $M_A$ has at most two connected components. $M_A$ is disconnected if and only if either the inequality (\ref{a}) or the
two inequalities (\ref{b1}) and (\ref{b2}) are satisfied.
\end{corollary}
\begin{corollary}
If $M_A$ is disconnected then for any fixed length for the $n$-th leg $\ell^-_n\le \ell_n\le \ell^+_n$, the manifold $M_\ell$
is disconnected
where $\ell=(\ell_1, \dots, \ell_{n-1}, \ell_n)$.
\end{corollary}
Recall that $M_\ell$ is defined as the moduli space of shapes of all closed planar $n$-gons with sides of lengths $\ell_1, \dots, \ell_n$.

\begin{corollary} If either $M_{\ell^+}$ or $M_{\ell^-}$ is connected then $M_A$ is connected.
\end{corollary}

One may restate Corollary \ref{cor6} in a different form:

\begin{corollary}\label{cor9} $M_A$ is disconnected if and only if there exist three indices $1\le i<j<k\le n$ such that for any $\ell_n\in [\ell_n^-, \ell_n^+]$ the pairs
$\{i,j\}$, $\{i, k\}$ and $\{j,k\}$ are long with respect to the length vector $\ell= (\ell_1, \dots, \ell_{n-1}, \ell_n)$.
\end{corollary}
\begin{proof} Indeed, in the case (a) the triple $i,j,k$ with the properties indicated above is given by $i=n-3, j=n-2, k=n-1$; in the case (b) we set $i=n-2, j=n-1, k=n$.
\end{proof}

This result is a generalization of the results of B. Jaggi \cite{J},
Theorem 4.1 of W. Lenhart and S. Whitesides \cite{LW}  and Theorem 1 from \cite{Ka1}; all results  mentioned above dealt with linkages with all legs having a fixed length.

}
\end{example}
\begin{example} \label{example3}
{\rm
Consider a two-dimensional example with both ends $M_{\ell^\pm}$ disconnected but $M_A$ connected.
Namely, let $n=4$ and $\ell_1=4$, $\ell_2=8$, $\ell_3=10$ and $\ell_4^+=12$, $\ell_4^-=1$. We see that both length vectors $(4,8,10,12)$ and $(4,8,10,1)$ determine disconnected one-dimensional manifolds $M_{\ell^+}\simeq M_{\ell^-}\simeq S^1\sqcup S^1$. Indeed, for the vector $\ell^+$ three indices $2, 3, 4$ form a \lq\lq rigid triple\rq\rq; for the vector $\ell^-$ a \lq\lq rigid triple\rq\rq\,  is formed by the indices $1, 2, 3$. Hence we see that $M_A$ is connected as the condition of Corollary \ref{cor9} is not satisfied. }
\end{example}

\begin{example}\label{example4}{\rm
In the case when $n=4$ the manifold $M_A$ has dimension two; it can be visualized as follows. Consider a planar quadrangle $ABCD$ as shown on Figure \ref{fig3}.
\begin{figure}[h]
\resizebox{5cm}{4cm} {\includegraphics[30,366][533,714]{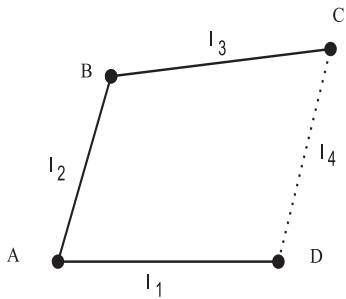}}
\caption{Variable quadrangle.}\label{fig3}
\end{figure}
The side $AD$ will remain horizontal and the side $CD$ represents the telescopic leg with its length $\ell_4$ varying between $\ell^-_4$ and $\ell^+_4$. We will assume below that
$\ell_1\le \ell_2\le \ell_3$.

First we disregard the condition that $|CD|$ should be within the interval $[\ell^-_4, \ell^+_4]$. Then we obtain that the position of the point $C$ must be within
the annulus with center at $A$ with exterior radius $R=\ell_2+\ell_3$ and interior radius $r=\ell_2-\ell_3$. Note that any internal point of this annulus is represented by exactly two configurations (which are symmetric to each other with respect to the line $AC$) while the boundary points are represented by a unique configuration of the bars $AB$ and $BC$ (since the boundary points of the annulus are achieved by collinear configurations).

Next we impose the condition that the distance $|CD|$ must satisfy $\ell_4^-\le |CD|\le \ell_4^+$. This means that $C$ must lie in another annulus with center $D$, external radius $\ell^+_4$ and internal radius $\ell^-_4$. One takes two copies of the intersection of the first and the second annuli and identifies the points lying on the boundary of the first annulus in both copies; the resulting space will be homeomorphic to $M_A$.

Consider now specifically the configuration space of the telescopic linkage with metric data as in Example \ref{example3}, i.e. $\ell_1=4$, $\ell_2=8,$ $ \ell_3=10,$ $\ell_4^-=1$, $\ell_4^+=12$. In this case the first annulus has radii $18$ and $2$ and the second annulus has radii $12$ and $1$ and the centers of the annuli are distance $4$ apart, as shown on Figure \ref{fig4}, a. On the right (Figure \ref{fig4}, b) one sees the intersection of these annuli (a disc with two disjoint small discs removed). To obtain $M_A$ one takes two copies of the intersection and glues them to each other along
boundary points of the first annulus (shown by bold on Figure \ref{fig4}.)
\begin{figure}[h]
\resizebox{6cm}{4cm} {\includegraphics[31,121][589,487]{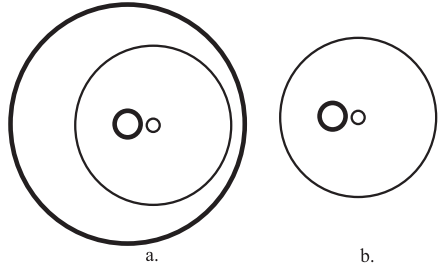}}
\caption{Two annuli (a) and their intersection (b).}\label{fig4}
\end{figure}
We obtain that in this example $M_A$ is homeomorphic to the sphere $S^2$ with four discs removed. In particular $M_A$ is connected although each of the boundary manifolds
$M_{\ell^\pm}$ is disconnected.

Let us compute in this example the numbers which appear in Theorem \ref{main}. One finds: $\alpha_0(\ell^-)=1$, $\beta_0(\ell^-, \ell^+)=0$, $\alpha_1(\ell^-)=3$, $\beta_1(\ell^+, \ell^-)=1$. Besides, $\alpha_0(\ell^+)=1$, $\beta_0(\ell^-, \ell^+)=0$, $\alpha_1(\ell^+)=1$, and $\beta_1(\ell^-, \ell^+)=0$. Thus, by Theorem \ref{main} the Betti numbers of $M_A$ are $1$ (in dimension 0) and $3$ (in dimension 1). This is consistent with our explicit description of the configuration space $M_A$ in this example.
}
\end{example}

\section{Proof of Proposition \ref{lm1} and Theorem \ref{main}}\label{proofs}

A robot arm is a simple planar mechanism consisting of several bars of fixed length connected by revolving joins as shown on Figure \ref{arm}.
\begin{figure}[h]
\resizebox{5cm}{4cm} {\includegraphics[90,368][440,632]{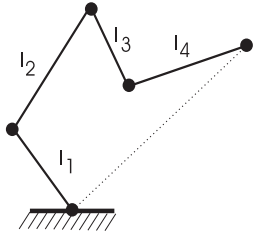}}
\caption{Robot arm.}\label{arm}
\end{figure}
We assume that there are $n-1$ bars of lengths $\ell_1\le \ell_2 \le \dots \le \ell_{n-1}$ and the initial point of the robot arm is fixed on the plane.
The space of all possible shapes of the arm
$$W=\{(u_1, \dots, u_{n-1})\in S^1 \times \dots\times  S^1\}/{\rm {SO(2)}}$$
is diffeomorphic to a torus of dimension $n-2$.

Consider the smooth function
\begin{eqnarray}
f: W \to \R, \quad f(u_1, \dots, u_{n-1}) = -\left|\sum_{i=1}^{n-1}\ell_i u_i\right|^2.
\end{eqnarray}
Geometrically, the value of $f$ equals the negative of the squared distance between the initial point of the arm to the end of the arm (shown by the dotted line on Figure \ref{arm}).

It is clear that the moduli space $M_A$ of the telescopic linkage is diffeomorphic to
the preimage $f^{-1}[a, b]$ where $$a= -(\ell^+_n)^2, \quad \mbox{and}\quad b= -(\ell^-_n)^2.$$
It is known that the critical points of $f$ are collinear configurations, see \cite{Fa2}, \cite{Fa3}.
The critical values of $f$ are of the form $-\left|\sum_{i=1}^{n-1}\ell_i \epsilon_i\right|^2$ where $\epsilon_i=\pm 1$. We obtain that $a$ and $b$ are regular values of $f$ if and only if the vectors $\ell^+=(\ell_1, \dots, \ell_{n-1}, \ell_n^+)$ and  $\ell^-=(\ell_1, \dots, \ell_{n-1}, \ell_n^-)$ are generic, i.e. they do not lie on the hyperplanes $S_J$, described in \S 2.
This implies Proposition \ref{lm1}. The orientability of $M_A$ follows from the orientability of $W$.

Next we prove Theorem \ref{main}.
We denote by $W^a=f^{-1}(-\infty, a]$ and $W^b=f^{-1}(-\infty, b]$ and $$W^{[a,b]}= f^{-1}[a,b].$$
Our goal is to compute the Betti numbers of $W^{[a,b]}\simeq M_A$.

Consider the homological exact sequence of the pair $(W^b, W^{[a,b]})$
$$\to H_{i+1}(W^b, W^{[a,b]})\to  H_i(W^{[a,b]})\to H_i(W^b) \stackrel{j_\ast}\to H_i(W^b, W^{[a,b]})\to $$
with coefficients\footnote{In this paper we will often not indicate explicitly the coefficient group understanding that it is the ring of integers $\Z$.} in $\Z$.
We may identify the relative homology as follows
$$H_i(W^b, W^{[a,b]})\simeq H_i(W^a, \partial W^a) \simeq H^{n-2-i}(W^a)\simeq \left(H_{n-2-i}(W^a)\right)^\ast.$$
Here we used the excision axiom, Poincar\'e duality and the universal coefficient theorem.
The last symbol
on the right denotes  the dual group $$\left(H_{n-2-i}(W^a)\right)^\ast = \Hom (H_{n-2-i}(W^a), \Z).$$ Note that the integral homology groups of $W^a$ and $W^b$ are free abelian, see \cite{Fa2}, \cite{Fa3}, which explains absence of the torsion term in the universal coefficient theorem.

Consider the intersection form
\begin{eqnarray}\label{intersection}
H_i(W^b) \otimes H_{n-2-i}(W^a) \to \Z
\end{eqnarray}
given geometrically by intersection of cycles in $W^b$. Note that $W^a\subset W^b$ and thus a cycle in $W^a$ can be viewed as a cycle in $W^b$.
It is well known that the
homomorphism
\begin{eqnarray}\label{dual1}H_i(W^b) \to \left( H_{n-2-i}(W^a)\right)^\ast\end{eqnarray}
associated to the bilinear form (\ref{intersection})
coincides with
\begin{eqnarray}\label{jj}
j_\ast: H_i(W^b) \to H_i(W^b, W^{[a,b]})
\end{eqnarray}
modulo the isomorphisms indicated above.

Let $k_i$ and $c_i$ denote the kernel and cokernel of the homomorphism (\ref{jj})  correspondingly. We obtain the short exact sequence
\begin{eqnarray}\label{exactseq}
0\to c_{i+1} \to H_i(M_A) \to k_i\to 0.
\end{eqnarray} It is clear that $k_i$ is free abelian and we will see below that $c_i$ is also torsion free for all $i$.
We denote by $r_i$ the rank of the intersection form
(\ref{intersection}). Then
$$r_i+\rk (k_i) = \rk (H_i(W^b)),$$
$$r_i + \rk (c_i) = \rk (H_i(W^b, W^{[a,b]}) )= \rk (H_{n-2-i}(W^a))$$
and the exact sequence (\ref{exactseq}) gives
\begin{eqnarray}\label{sum}
\begin{array}{c}\rk (H_i(M_A)) \, =\, \rk (H_i(W^{[a,b]}))  =\\   \\ \rk (H_i(W^b)) + \rk (H_{n-3-i}(W^a)) - r_i - r_{i+1}.\end{array}
\end{eqnarray}
It also follows that $H_i(M_A)$ is torsion free if and only if $c_{i+1}$ is torsion free.

Next we describe homology of the manifolds $W^a$ and $W^b$ following \cite{Fa2}, \cite{Fa3}.
For any subset $J\subset \{1, \dots, n-1\}$ consider the subset $W_J\subset W\simeq T^{n-2}$ consisting of all configuration $(u_1, \dots, u_{n-1})$ such that $u_i=u_j$ for all $i, j\in J$.
In other words, we \lq\lq freeze\rq\rq\,  all links labeled by indices in $J$ to be parallel to each other. It is clear that $W_J$ is diffeomorphic to a torus of dimension $n-1-|J|$.

The torus $W_J$ is contained in $W^a$, i.e. $W_J\subset W^a$, if and only if $J$ (viewed as a subset of $\{1, \dots, n\}$) is long with respect to $\ell^+$.
Indeed, let $p_J=(u_1, \dots, u_{n-1})$ be the configuration where $u_i=1$ for all $i\in J$ and
$u_i=-1$ for all $i\notin J$. Then the maximum of the restriction $f|W_J$ is either $0$ or $f(p_J)$, see \cite{Fa2}, Lemma 8, statement (4). The inequality $f(p_J)\le a$ is equivalent to
$\langle \ell^+, \epsilon_J\rangle >0$ which means that $J$ is long with respect to $\ell^+$.

By Lemma 9 from \cite{Fa2} the homology classes of the submanifolds $W_J$
form a basis of the homology vector space $H_i(W^a)$ where $J$ runs over all subsets $J\subset \{1, \dots, n-1\}$ of cardinality $n-1-i$ which are long with respect to $\ell^+$.
Thus using the notation introduced earlier one obtains
\begin{eqnarray}\label{15}
\rk \, H_i(W^a) = \alpha_i(\ell^+).\end{eqnarray}

Similarly, for a subset $I\subset \{1, \dots, n-1\}$ one has $W_I\subset W^b$ if and only if $I$ is long with respect to $\ell^-$. The homology $H_{i}(W^b)$ is freely generated by homology classes of all submanifolds
$W_I\subset W$ where $I$ runs over all subsets $I\subset \{1, \dots, n-1\}$ of cardinality $n-1-i$ which are long with respect to $\ell^-$. We have
\begin{eqnarray}\label{16}\rk \, (H_{i}(W^b)) = \alpha_{i}(\ell^-).\end{eqnarray}

Next we have to analyze the intersection form (\ref{intersection}) in the basis of homology given by the submanifolds $W_I$. For this purpose we represent
$H_i(W^b) $ as a direct sum
\begin{eqnarray}\label{17}
H_i(W^b) = A^b_i\oplus B^b_i \oplus C^b_i,\end{eqnarray}
described below.
The group $A^b_i$ is generated by the homology classes $[W_I]$ with those subsets $I\subset \{1, \dots, n-1\}$, $|I|=n-1-i$, which are long with respect to $\ell^-$ and such that $\hat I$ is long with respect to $\ell^+$. Here $\hat I$ denotes the subset of $\{1, \dots, n\}$ which is obtained from $I$ by removing the maximal index lying in $I$ and adding $n$.
Similarly, $B^b_i$ is generated by the homology classes $[W_I]$ with those subsets $I\subset \{1, \dots, n-1\}$, $|I|=n-1-i$, which are long with respect to $\ell^-$ and such that $n-1\in I$ and $\hat I$ is short with respect to $\ell^+$; note that in this case $\hat I$ is obtained from $I$ by deleting $n-1$ and adding $n$. Finally, $C_i^b$ is generated by the homology classes $[W_I]$ with $I\subset \{1, \dots, n-2\}$, $|I|=n-1-i$, which is long with respect to $\ell^-$ and such that $\hat I$ is short with respect to $\ell^+$.

We represent the group $H_i(W^a)$ as a direct sum in a similar fashion
\begin{eqnarray}\label{18}
H_i(W^a) = A^a_i\oplus B^a_i \oplus C^a_i,\end{eqnarray}
where $A^a_i, B^a_i, C^a_i$ are defined analogously to $A^b_i, B^b_i, C^b_i$ with
the roles of $\ell^+$ and $\ell^-$
interchanged. In more detail, $A^a_i$ is generated by the homology classes $[W_J]$ with $J\subset \{1, \dots, n-1\}$, $|J|=n-1-i$, which is long with respect to $\ell^+$ and such that $\hat J$ is long with respect to $\ell^-$.
The space $B^a_i$ is generated by the homology classes $[W_J]$ with $J\subset \{1, \dots, n-1\}$, $|J|=n-1-i$, $n-1\in J$, which are long with respect to $\ell^+$ and such that $\hat J$ is short with respect to $\ell^-$. Finally, $C_i^a$ is generated by the homology classes $[W_J]$ with $J\subset \{1, \dots, n-2\}$, $|J|=n-1-i$, long with respect to $\ell^+$ and such that $\hat J$ is short with respect to $\ell^-$.

Note that in decompositions (\ref{17}) and (\ref{18}) each of the subgroups has a specified basis which will be important in the sequel.
Counting the number elements in the basis we obtain
\begin{eqnarray}\label{19}
\rk (B_i^b) = \beta_i(\ell^+, \ell^-), \quad \rk (B_i^a) = \beta_i(\ell^-, \ell^+),
\end{eqnarray}
according to our definitions. We see that the statement of Theorem \ref{main} would follow from (\ref{dual}), (\ref{sum}), (\ref{15}), (\ref{16}), (\ref{19}) once it is shown that
the cokernel $c_i$ of the intersection form (\ref{intersection}) has no torsion and the rank of the intersection from  (\ref{intersection}) equals $\rk (B_i^b)$.

Suppose that $I\subset \{1, \dots, n-1\}$ is a subset of cardinality $n-i-1$ which is long with respect to $\ell^-$ and
$J\subset \{1, \dots, n-1\}$ is a subset of cardinality $i+1$ which is long with respect to $\ell^+$. Then the homology classes
$$[W_I]\in H_i(W^b), \quad [W_J]\in H_{n-2-i}(W^a)$$ of the submanifolds $W_I$ and $W_J$ (properly oriented) have complementary dimensions and one wants to compute their intersection via (\ref{intersection}).
By formula (33) from \cite{Fa2} %and noting that we are dealing with $\Z_2$ coefficients we find that
\begin{eqnarray}\label{formula}
[W_I]\cdot [W_J]= \left\{
\begin{array}{lll}
\pm 1, & \mbox{if} & |I\cap J|=1, \\
0, & \mbox{if} & |I\cap J|>1.
\end{array}\right.
\end{eqnarray}

To make this more precise we fix orientations of $W$ and all submanifolds $W_J$ as follows.
Recall that $W$ is the quotient of $T^{n-1}$ by the diagonal action of
${\rm {SO}}(2)$. Let $e_i$ denote the unit tangent vector field on $T^{n-1}$ which is tangent to the $i$-th circle and rotates it in the positive direction, where $i=1, \dots, n-1$.
Let $e_i'$ be the image of $e_i$ under the projection $T^{n-1}\to W$. The fields $e_1', \dots, e'_{n-1}$ generate the tangent space to $W$ at every point and satisfy the relation
$e'_1+\dots+e'_{n-1}=0$.
We orient $W$ by declaring the basis $e_2',  e'_3, \dots, e'_{n-1}$ to be positive.

Consider now a subset $I\subset \{1, \dots, n-1\}$ and the corresponding submanifold $W_I$. Let $\bar I=\{i_1<i_2< \dots<i_r\}$ denote the complement of $I$, where
$r=n-1-|I|$. Then the fields $e'_{i_1}, \dots, e'_{i_r}$ form a basis of the tangent space to $W_I$ at every point and we orient $W_I$ according to the basis
$e'_{i_1}, \dots, e'_{i_r}$.

The following statement is a refinement of the first part of formula (\ref{formula}). It is presented here only of the sake of completeness as it will not be used in the proof of Theorem \ref{main}:

\begin{lemma} Suppose that $I, J \subset \{1, \dots, n-1\}$ are such that $I\cap J=\{j\}$ and $I\cup J = \{1, \dots, n-1\}$. Then, with the orientations specified as indicated above, one has
\begin{eqnarray}
[W_I]\cdot [W_J]= (-1)^{j+1}\epsilon_j(\bar I, \bar J),
\end{eqnarray}
where $\epsilon_j(\bar I, \bar J)$ denotes the sign of the permutation of the set $$\{1, \dots, n-1\}-\{j\}$$ determined by placing all elements of $\bar I$ in their natural ordering and then all elements of $\bar J$ in their natural ordering.
\end{lemma}
\begin{proof} We know from \cite{Fa3}, page 27, that the submanifolds $W_I$ and $W_J$ intersect transversally at a single point and we need to determine the sign of this intersection. Let $\bar I=\{i_1< \dots<i_r\}$ and $\bar J=\{j_1< \dots<j_s\}$ where $r=n-1-|I|$ and $s=n-1-|J|$. Note that $r+s=n-2$.
The tangent space to $W_I$ is freely generated by the vector fields $e'_{i_\alpha}$ (where $\alpha=1, \dots, r$) and the tangent space to $W_J$ is freely generated by the fields $e'_{j_\beta}$ (where $\beta=1, \dots, s$). Thus the intersection number $[W_I]\cdot [W_J]$ equals $\pm 1$ depending on whether the orientation of $W$ determined by the
the basis $e'_{i_1}, \dots, e'_{i_r}, e'_{j_1}, \dots, e'_{j_s}$ is positive or negative. Thus, we obtain that
$[W_I]\cdot [W_J]= \epsilon_j(\bar I, \bar J)\cdot \eta_j$ where $\eta_j$ denotes the sign of the base obtained from the set of vector fields $e'_1, \dots, e'_{n-1}$ by removing the field $e_j'$.
Since $e'_1+\dots+e'_{n-1}=0$ it is easy to see that $\eta_j=(-1)^{j+1}$.
\end{proof}

Consider decomposition (\ref{17}) in dimension $i$ as well as decomposition (\ref{18}) in the dual dimension $$i'=n-2-i.$$
 Suppose that $[W_I]\in A_i^b$. The intersection $[W_I]\cdot [W_J]\in \Z$ is nonzero only if $J$ is obtained from the complement of $I$ in the set $\{1, \dots, n-1\}$ by adding an element of $I$.
Can such $J$ be long with respect to $\ell^+$?
If $J$ with these properties exists then its complement $\tilde J$ in $\{1, \dots, n\}$ is short with respect to $\ell^+$. But $\tilde J$ is obtained from $I$ by removing one element and adding $n$. It follows that the set $\hat I$ obtained from $I$ by removing the largest element from $I$ and adding $n$ is also short with respect to $\ell^+$. However this is impossible according to our definition of $A_i^b$. Hence we obtain that for any $[W_I]\in A_i^b$
and for any $[W_J]\in H_{i'}(W^a)$ one has $[W_I]\cdot [W_J]=0$.

Similarly one obtains that for any $[W_J]\in A_{i'}^a$
and for any $[W_I]\in H_{i}(W^b)$ one has $[W_I]\cdot [W_J]=0$.

Consider now $[W_I]\in B_i^b$ and $[W_J]\in B_{i'}^a$. Since the sets $I$ and $J$ both contain $n-1$ the intersection $[W_I]\cdot [W_J]\not=0$ iff $I\cap J=\{n-1\}$, i.e. when $J$ is obtained
from the complement $\tilde I$ by removing $n$ and adding $n-1$. We see that given $[W_I]\in B_i^b$ there exists a unique basis element $[W_J]\in B_{i'}^a$ such that
$[W_I]\cdot [W_J]=\pm 1$. In particular, the restriction of the intersection form (\ref{intersection}) onto $B_i^b\otimes B_{i'}^a$ is nondegenerate and
$$\rk (B_i^b) = \beta_i(\ell^+, \ell^-) = \rk (B_{i'}^a)= \beta_{i'}(\ell^-, \ell^+).$$

As another remark we mention that $[W_I]\cdot [W_J]=0$ if $[W_I]\in C_i^b$ and $[W_J]\in C_{i'}^a$. Indeed in this case the sets $I, J\subset \{1, \dots, n-2\}$ must have at least two elements in common, $|I\cap J|>1$, since $|I|=n-i-1$ and $|J|= i+1$.

For each basis element $[W_I]\in C_i^b$ define
\begin{eqnarray}\label{fourier}
Y_I = [W_I] - \sum_{K}
\frac{[W_I]\cdot [W_{K'}]}{[W_K]\cdot [W_{K'}] }
\cdot [W_K]\, \in\,  H_i(W^b),
\end{eqnarray}
where $[W_K]$ runs over all basis elements of $B_i^b$ and $K'$ stands for $$K'=\tilde K -\{n\} \cup \{n-1\}.$$ In the last formula $\tilde K$ denotes the complement of $K$ in $\{1, \dots, n\}$. This class $Y_I$ has clearly the property that the intersection
\begin{eqnarray}\label{vanish}
Y_I\cdot [W_J]=0\end{eqnarray}
 is trivial for all $[W_J]\in A_{i'}^a\oplus B_{i'}^a$. Next we show that vanishing (\ref{vanish}) holds also for $[W_J]\in C^a_{i'}$.

 With this goal in mind we first rewrite formula (\ref{fourier}) retaining only nonzero terms, i.e. only terms with $|I\cap K'|=1$. We obtain that the nonzero terms in (\ref{fourier}) correspond to subsets $K$ of the form $$K=I-\{i\}\cup \{n-1\}=K_i$$ where $i\in I$.
 Assuming that $I\subset \{1, \dots, n-2\}$ is long with respect to $\ell^-$ and $\hat I$ is short with respect to $\ell^+$ one obtains that for any $i\in I$ the set $K_i$ is long with respect to $\ell^-$ and the set $\hat K_i$ is short with respect to $\ell^+$ (for obvious reasons). Thus we have
 \begin{eqnarray}\label{fourier1}
Y_I = [W_I] - \sum_{i\in I}\, \frac{[W_I]\cdot [W_{K'_i}]}{[W_{K_i}]\cdot [W_{K_i'}] } [W_{K_i}].
\end{eqnarray}

Given $[W_J]\in C_{i'}^a$ consider the intersection $Y_I\cdot [W_J] $ which equals
\begin{eqnarray*} \label{last}
\begin{array}{c} [W_I]\cdot[W_J] - \sum_{i\in I} \, \displaystyle{\frac{[W_I]\cdot [W_{K'_i}]}{[W_{K_i}]\cdot [W_{K'_i}]}}\cdot \left([W_{K_i}]\cdot [W_J]\right) =
 \\ \\  - \sum_{i\in I} \, \displaystyle{\frac{[W_I]\cdot [W_{K'_i}]}{[W_{K_i}]\cdot [W_{K'_i}]}}\cdot \left([W_{K_i}]\cdot [W_J]\right).\end{array}
 \end{eqnarray*}
If for some $i\in I$ one has $|J\cap K_i|=1$ then $|I\cap J|=2$.
Thus we obtain that if $|I\cap J|>2$ then all terms in the above formula are trivial and therefore $Y_I\cdot [W_J] =0$.

Assuming that $|J\cap J|=2$, say, $I\cap J=\{i, j\}$, we obtain that
\begin{eqnarray}
 Y_I\cdot [W_J] =-\mu_i -\mu_j,
\end{eqnarray}
where
$$\mu_i = \frac{[W_I]\cdot [W_{K'_i}]}{[W_{K_i}]\cdot [W_{K'_i}]}\cdot \left([W_{K_i}]\cdot [W_J]\right)$$
and $\mu_j$ is defined similarly with $j$ instead of $i$. We show below that $\mu_i+\mu_j=0$ and hence $Y_I\cdot [W_J] =0$ for any $[W_J]\in H_{i'}(W^a)$.

Consider the homeomorphism $T^{n-1}\to T^{n-1}$ interchanging the $i$-th and the $j$-th coordinates.
It descends to a homeomorphism $\phi: W\to W$. Since the subsets $I$ and $J$ both contain $i$ and $j$ it follows that $\phi(W_I)=W_I$ and $\phi(W_J)=W_J$.
Besides, $\phi(W_{K_i})=W_{K_j}$ and $\phi(W_{K_j})=W_{K_i}$; moreover, $\phi(W_{K'_i})=W_{K_j'}$ and $\phi(W_{K'_j})=W_{K_i'}$.

Note that $\phi$ reverses the orientation of $W$  and therefore for any two homology classes $z\in H_i(W)$, $z'\in H_{i'}(W)$ one has
\begin{eqnarray}\label{sign0}
\phi_\ast(z)\cdot \phi_\ast(z') = - z\cdot z'.
\end{eqnarray}
Besides, $\phi$ preserves the orientations of the submanifolds
$W_I$ and $W_J$ and hence
\begin{eqnarray}\phi_\ast[W_I]=[W_I], \quad \phi_\ast[W_J]=[W_J].
\end{eqnarray}

Using our convention concerning orientations of the submanifolds $W_J$ and assuming that $i<j$, one obtains
\begin{eqnarray}\label{sign1}
\phi_\ast[W_{K_i}]= (-1)^{|(i,j)\cap \bar I|}\cdot [W_{K_j}].
\end{eqnarray}
Here $|(i,j)\cap \bar I|$ is the number of integers between $i$ and $j$ which do not belong to $I$.
Similarly,
\begin{eqnarray}\label{sign2}
\phi_\ast[W_{K_j}]= (-1)^{|(i,j)\cap \bar I|}\cdot [W_{K_i}].
\end{eqnarray}
Analogously, we have
\begin{eqnarray}\label{sign3}
\begin{array}{l} \phi_\ast[W_{K'_i}]= (-1)^{|(i,j)\cap I|}\cdot [W'_{K_j}],\\  \\ \phi_\ast[W_{K'_j}]= (-1)^{|(i,j)\cap  I|}\cdot [W_{K'_j}].\end{array}
\end{eqnarray}

Therefore, using (\ref{sign0}) - (\ref{sign3}), we obtain
\begin{eqnarray*}\begin{array}{c}
\mu_i = - \displaystyle{\frac{[\phi (W_I)]\cdot [\phi (W_{K'_i})]}{[\phi (W_{K_i})]\cdot [\phi (W_{K'_i})]}}\cdot \left([\phi (W_{K_i})]\cdot [\phi (W_J)]\right)= \\ \\
- \displaystyle{\frac{[W_I]\cdot [W_{K'_j}]}{[W_{K_j}]\cdot [W_{K'_j}]}}\cdot \left([W_{K_j}]\cdot [W_J]\right) =  -\mu_j. \end{array}
\end{eqnarray*}
All signs which come from formulas (\ref{sign1}), (\ref{sign2}), (\ref{sign3}) cancel each other since each of them appears twice. Thus, $\mu_i+\mu_j=0$ and
$$Y_I\cdot [W_J]=0 \quad \mbox{for all}\quad  [W_J]\in C_{i'}^a.$$

Now we are able to complete the proof of Theorem \ref{main}. Denote by $D_i^b$ the subgroup freely generated by the homology classes $Y_I$ where the subset $I\subset \{1, \dots, n-2\}$ is such that
$[W_I]\in C_i^b$. We have a direct sum decomposition
$$H_i(W^b) = A_i^b\oplus B_i^b\oplus D_i^b$$ and the homomorphism
$$j_\ast:H_i(W^b) \to H_i(W^b, W^{[a,b]})=(H_{i'}(W^a))^\ast$$
vanishes on $A_i^b\oplus D_i^b$. However the restriction $j_\ast|B_i^b$ is a monomorphism onto a direct summand (since its composition with the projection
$\left(H_{i'}(W^a)\right)^\ast \to (B_{i'}^a)^\ast$ is an isomorphism).
We obtain that the cokernel $c_i$ of $j_\ast$ is torsion free and the rank of the image of $j_\ast$ equals $$r_i =\rk (B_i^b)=\rk (B_{i'}^a).$$

Theorem \ref{main} now follows from (\ref{dual}), (\ref{sum}), (\ref{15}), (\ref{16}), (\ref{19}). \qed %This completes the proof of Theorem \ref{main}.

\section{Equilateral linkage with a telescopic leg}

In this section as an illustration of Theorem \ref{main} we examine the special case when all bars of the linkage have length $1$ and the length of the telescopic leg may vary in an interval $[a, b]$ where $0<a<b$.

Using the previously introduced notations we have in this case $$\ell_1=\dots=\ell_{n-1}=1,\quad \ell_n^-=a,\quad \ell_n^+=b.$$
The metric data of the linkage is not generic if and only if either $a$ or $b$ is an integer of opposite parity to $n$.
For example, if $n$ is even then the genericity assumption is satisfied if neither $a$ nor $b$ is an odd integer. If $n$ is odd then we require that neither $a$ nor $b$ can be an even integer.

Let us compute the numbers $\alpha_k(\ell^\pm)$ and $\beta_k(\ell^\pm, \ell^\mp)$ which appear in Theorem \ref{main}.

A subset  $J\subset \{1, \dots, n-1\}$ of cardinality $n-1-k$ is long with respect to $\ell^-$ if and only if
$a<n-2k-1$. Hence we obtain
\begin{eqnarray}
\alpha_k(\ell^-)= \left\{ \begin{array}{ll}
{\binom{n-1} {k}}, & \mbox{if $a< n-2k-1$,}\\ \\
0, & \mbox{if $a\ge n-2k-1$}.
\end{array}\right.
\end{eqnarray}

Similarly, one computes explicitly the numbers $\beta_k(\ell^+, \ell^-)$. Simple analysis shows that $\beta_k(\ell^+, \ell^-)$ and $\beta_k(\ell^-, \ell^+)$
can be nonzero only in the case when
$n$ is even, $n=2r+2$, and $k=r$, i.e. when one considers the middle dimensional homology. In this case one has
\begin{eqnarray}\qquad
\beta_r(\ell^+,\ell^-)= \beta_r(\ell^-,\ell^+) \, =\, \left\{ \begin{array}{ll}
{\binom{2r} {r}}, & \mbox{if $a< 1$ and $b<1$,}\\ \\
0, & \mbox{otherwise}.
\end{array}\right.
\end{eqnarray}

Thus for $2k<n-4$ one has
\begin{eqnarray}
\rk H_k(M_A) = \left\{
\begin{array}{ll}
{\binom {n-1} k}, & \mbox{if $a<n-2k-1$},\\ \\
0, &\mbox{if $a>n-2k-1$}.
\end{array}\right.
\end{eqnarray}
Hence homology in low dimension does not depend on the value of the parameter $b$. Similarly one obtains that for $2k>n-2$ the $k$-dimensional Betti number equals
\begin{eqnarray}
\rk H_k(M_A) = \left\{
\begin{array}{ll}
{\binom {n-1} {k+2}}, & \mbox{if $b<2k-n +5$},\\ \\
0, &\mbox{if $b> 2k-n+5$}.
\end{array}\right.
\end{eqnarray}

It remains to calculate the Betti numbers in the middle dimension, i.e. for $n-2k$ equal $2, 3, 4$.

For $n-2k=2$ or $n-2k=3$ we have $\beta_{k+1}(l^{+},l^{-})=0$. In the first case $n-2k=2$ we find
$$ \rk H_{k}(M_{A}) =
\begin{cases} \binom{n-1}{k}+\binom{n-1}{k+2}-\binom{n-2}{k}, & \mbox{if } a<1,\, \, b<1, \\
\binom{n-1}{k}+\binom{n-1}{k+2}, & \mbox{if } 1<b<3,\, \, a<1, \\
\binom{n-1}{k+2},& \mbox{if} 1<b<3, \, \, 1<a,\\
\binom{n-1}{k}, & \mbox{if } b>3,\, \,  a<1, \\
0, & \mbox{if } b>3, \, \, 1< a.\end{cases}$$
In the case $n-2k=3$ we have $\beta_{k}(l^{+},l^{-})=0$ and thus
$$\rk H_{k}(M_{A}) =
\begin{cases}
\binom{n-1}{k}+\binom{n-1}{k+2}, & \mbox{if } b<2, \\
\binom{n-1}{k}, & \mbox{if } a<2<b \\ 0, & \mbox{if } a,b>2. \end{cases} $$

Finally, let us consider the case $n-2k=4$. Here we have $\beta_{k}(l^{+},l^{-})=0$. If $b<1$ then $\beta_{k+1}(l^{+},l^{-})=\beta_{n-k-3}(l^{-},l^{+})$ is non-zero and
we have
$$\rk H_{k}(M_{A})= \binom{n-1}{k}+\binom{n-1}{k+2}-\binom{n-2}{k+1}.$$
For $n-2k=4$ and $b>1$ we have
$$\rk H_{k}(M_{A}) =
\begin{cases} \binom{n-1}{k}, & \mbox{if } a<3, \\ 0, & \mbox{if } a>3.  \end{cases} $$

This can be compared with the Betti numbers of equilateral linkages with no telescopic leg, see \cite{Fa2}, \cite{Fa3}.

\section{The disconnected case}

In this section we prove the following statement which is a generalization of a result of M. Kapovich and J. Millson \cite{Ka1} who dealt with non-telescopic linkages.

\begin{proposition}\label{prop11}
If $M_A$ is disconnected then it is diffeomorphic to the product
$$[0,1]\times (T^{n-3}\sqcup T^{n-3})$$
of the interval $[0,1]$ and the disjoint union of two copies of ${(n-3)}$-dimensional torus $T^{n-3}$.
\end{proposition}

First we prove an analogue of Corollary \ref{cor2} from section \ref{sec2} involving a small non-telescopic leg.
Results of this type are known for the usual (non-telescopic) linkages (J.-Cl. Hausmann).

\begin{lemma} \label{lemma12} Consider a planar linkage with a telescopic leg which has generic metric data $A$ given by $\ell_1\le \dots \le \ell_{n-1}$ and $0<\ell^-_n<\ell_n^+$.
Suppose that $\ell_1>0$ is so small that the following is true: for $\ell_n=\ell_n^\pm$ and for any choice of $\epsilon_2=\pm 1, \dots, \epsilon_n=\pm 1$ such that
$$\sum_{i=2}^n\epsilon_i\ell_i >0$$ one has $$\sum_{i=2}^n\epsilon_i\ell_i >\ell_1.$$ Then $M_A$ is diffeomorphic to $$M_{A'}\times S^1,$$ where
$A'$ is the metric data of the linkage having $n-2$ legs of fixed lengths $\ell_2\le \dots\le \ell_{n-1}$ and a telescopic leg of length varying in the interval $[\ell_n^-, \ell_n^+]$.
\end{lemma}

\begin{proof} Let $V=T^{n-2}\times [\ell_n^-, \ell_n^+]$ denote the product of a torus of dimension $n-2$ and of interval. The points of $V$ are of the form
$(u_3, \dots, u_n, \ell_n)$ where $u_3, \dots, u_n\in S^1\subset \R^2$ are unit vectors on the plane and $\ell_n$ is a number which belongs to the interval $[\ell_n^-, \ell_n^+]$.
Consider a smooth map $g: V\to \R^2$ given by
\begin{eqnarray}\label{gg}
g(u_3, \dots, u_n, \ell_n) = \sum_{i=2}^n \ell_i u_i\in \R^2;
\end{eqnarray}
in this formula $u_2$ denotes the unit vector pointing in the direction of the $x$-axis. Note that $g^{-1}(0)$ coincides with the configuration space $M_{A'}$ of the telescopic linkage with sides
$\ell_2, \dots, \ell_{n-1}$ and with telescopic leg with parameters $0<\ell_n^-<\ell_n^+$.

Now, let $C\subset \R^2$ denote the circle with center at the origin and with radius $\ell_1$. Then the preimage $g^{-1}(C)$ is the configuration space $M_A$.
\begin{figure}[h]
\resizebox{4cm}{5cm} {\includegraphics[13,181][544,799]{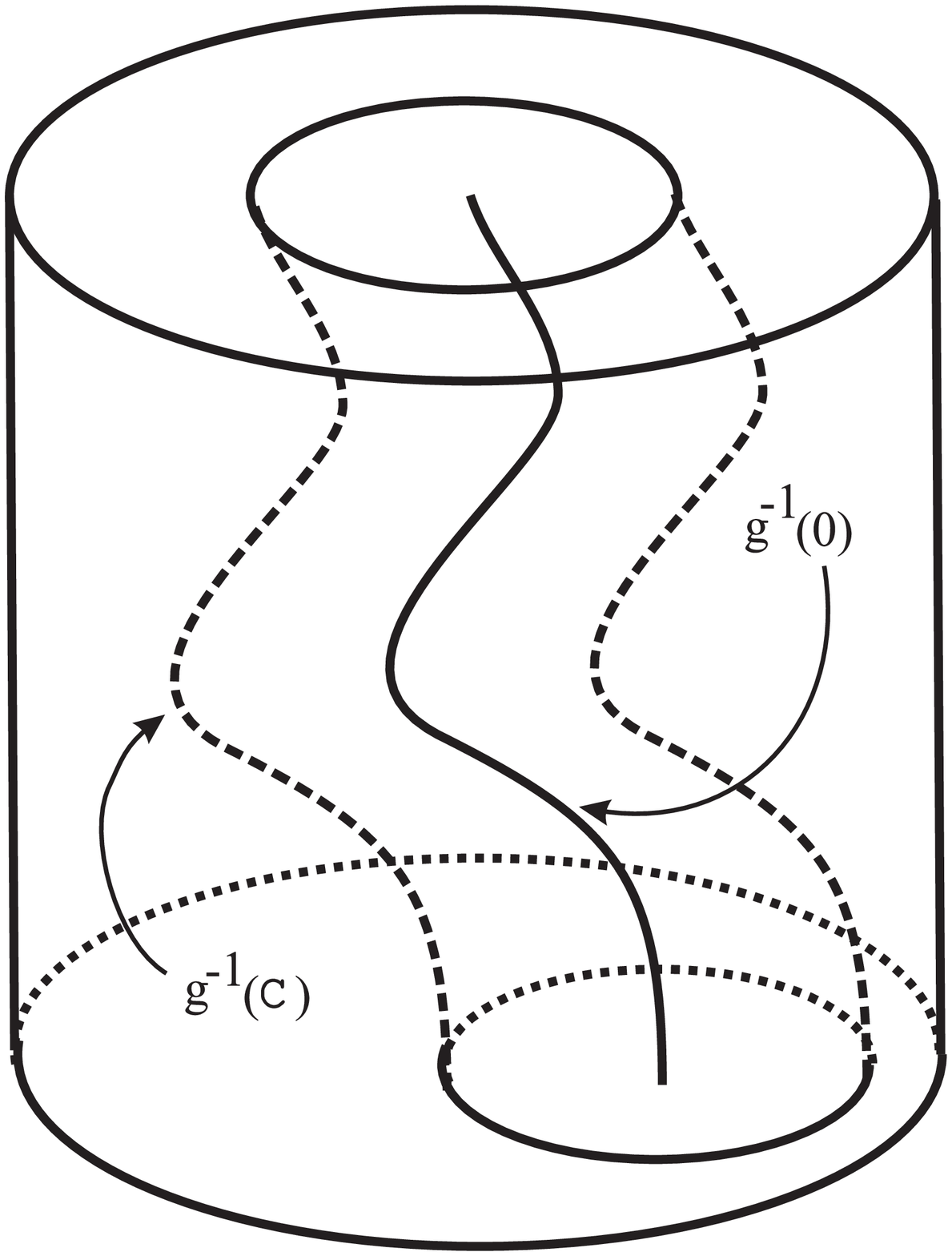}}
\caption{Manifold $V$ and submanifolds $g^{-1}(0)$ and $g^{-1}(C)$.}\label{vvv}
\end{figure}

Note that $g$ is a submersion as already the last summand in (\ref{gg}) has surjective differential. The boundary $\partial V$ has two components $\partial_-V= T^{n-2}\times \ell_n^-$ and
$\partial_+V= T^{n-2}\times \ell_n^+$.
The critical points of the restriction $g|(\partial _\pm V)$ are collinear configurations lying entirely in the $x$-axis. Our assumption on $\ell_1$ guarantees that the image of any of the critical points of $g|(\partial_\pm V)$ lies outside the circle $C$. Thus we see that $g$ is a submersion over the disk bounded by $C$ and therefore
$g^{-1}(C) $ is diffeomorphic to $g^{-1}(0)\times C$, see Figure \ref{vvv}. This completes the proof.
\end{proof}

\begin{proof}[Proof of Proposition \ref{prop11}] Assume that the metric data $A$ is given by the numbers $\ell_1\le \dots\le \ell_{n-1}$ and $0<\ell_n^-<\ell_n^+$. If $M_A$ is disconnected we may apply Corollary \ref{cor9} asserting that there exist three indices $1\le i<j<k\le n$ such that the three pairs $\{i,j\}$, $\{i, k\}$ and $\{j,k\}$ are long with respect to $(\ell_1, \dots, \ell_{n-1}, \ell_n)$ for any $\ell_n\in [\ell_n^-, \ell_n^+]$.

There are two possibilities: either (a) the triple $\{i,j,k\}$ does not contain $n$, the index of the telescopic leg, or (b) $n=k$.

Consider first the case (a). Then obviously $i=n-3$, $j=n-2$ and $k=n-1$. Let us show that we may apply Proposition \ref{cor2}. Indeed, a subset $J\subset \{1, \dots, n\}$ is long with respect to $\ell^-$ if and only if it contains at least two of the indices $\{n-3, n-2, n-1\}$. In particular, for a subset $J$ the property of being short or long with respect to
$\ell^-$ does not depend on whether $J$ contains elements $i<n-3$. We trivially obtain
$$[\ell^-]=\ell_{n-2}+\ell_{n-3} - \ell_1-\dots-\ell_{n-4} - \ell_{n-1}-\ell_n^-;$$
see (\ref{five}) for the notation $[\ell^-]$. We see that inequality (\ref{ineq3}) is equivalent to
$$\ell_{n-2}+\ell_{n-3} - \ell_1-\dots-\ell_{n-4} - \ell_{n-1}-\ell_n^+>0,$$
which is valid since $\{n-2, n-3\}$ is long with respect to $\ell^+.$ By Proposition \ref{cor2} we have $M_A \simeq M_{\ell^-}\times [0,1]$ and clearly $M_{\ell^-}$ is disconnected. Now we may refer to \cite{Ka1} for the statement that $M_{\ell'}$ is diffeomorphic to $T^{n-3}\sqcup T^{n-3}$ and Proposition \ref{prop11} follows.

Consider now the case (b). Then $i=n-2$, $j=n-1$, $k=n$. Proposition \ref{prop11} is trivial for $n=3$ hence we will assume that $n>3$. A subset $J\subset \{1, \dots, n\}$ is long with respect to $(\ell_1, \dots, \ell_{n-1}, \ell_n)$ where $\ell_n^-\le \ell_n\le \ell_n^+)$ if and only if it contains at least two indices out of $\{n-2, n-1, n\}$.
Again, the property of a subset $J$ to be short or long with respect to
$\ell$ does not depend on whether $J$ contains elements which are less than $n-3$.
Hence
$\sum_{i=2}^n\epsilon_i\ell_i>0$ implies %$\ell_1+ \sum_{i=2}^n\epsilon_i\ell_i>0$ which implies
$ \sum_{i=2}^n\epsilon_i\ell_i>\ell_1$. We see that Lemma \ref{lemma12} is applicable and $M_A$ is diffeomorphic to the product $M_{A'}\times S^1$ where $A'$ is the metric data of a linkage with legs of fixed lengths $\ell_2\le \dots\le \ell_{n-1}$ and with a telescopic leg with parameters $\ell_n^-<\ell_n^+.$ Proposition \ref{prop11} now follows by induction as $M_{A'}$ must be disconnected.

\end{proof}

\vspace{2ex}

\end{document}